\documentclass[10pt]{amsart}

\usepackage{amsmath, amsfonts, amssymb, amsthm, graphics}

\newtheorem{theorem}{Theorem}[section]
\newtheorem{lemma}[theorem]{Lemma}
\newtheorem{corollary}[theorem]{Corollary}

\theoremstyle{definition}

\newtheorem{remark}[theorem]{Remark}

\numberwithin{equation}{section}

\begin{document}

\title[Regularly oscillating mappings]{Regularly oscillating mappings between metric spaces and a theorem of Hardy and Littlewood}

\author{Marijan Markovi\'{c}}

\begin{abstract}
This paper is motivated by the classical theorem  due  to  Hardy  and Littlewood which concerns analytic mappings on the unit disk and relates the growth of the derivative with  the H\"{o}lder
continuity. We obtain  a version of this  result  in a very  general setting -- for regularly oscillating mappings on a metric space equipped  with a weight, which is a continuous and positive
function, with values in another  metric space. As a consequence, we derive the Hardy and Littlewood theorem for analytic mappings on the  unit  ball of a  normed  space.
\end{abstract}

\address{
Faculty of Sciences and Mathematics\endgraf
University of Montenegro\endgraf
D\v{z}ord\v{z}a Va\v{s}ingtona bb\endgraf
81000 Podgorica\endgraf
Montenegro\endgraf}

\email{marijanmmarkovic@gmail.com}

\subjclass[2020]{Primary 26A16, 51F30, 46E50; Secondary 46G20, 30L15, 54E45}

\keywords{Hardy  and Littlewood theorem, regularly oscillating mappings, the distance function, uniform domains, mappings between metric spaces, analytic       mappings   between normed spaces,
Lipschitz type spaces, Bloch type spaces}

\maketitle

\section{Introduction}

Let   $\mathbb{D}\subseteq\mathbb{C}$ be the unit disk,   and $\alpha\in (0,1]$. The classical Hardy and Littlewood theorem \cite{HARDY.MZ} states that an analytic mapping $f$  on $\mathbb{D}$
satisfies
\begin{equation}\label{EQ.HL.1}
| f'(z) |\le C_f(1- |z|)^{\alpha-1},\quad  z\in \mathbb{D},
\end{equation}
if  and only if it is H\"{o}lder continuous, i.e.,
\begin{equation}\label{EQ.HL.2}
|f(z) - f (w)| \le \ C'_f |z-w|^\alpha,\quad z,w\in\mathbb{D},
\end{equation}
where  $C_f$ and $C'_f$ are constants (may depend only on the mapping  $f$). This theorem  may be found in  the Duren book \cite[Theorem 5.1]{DUREN.BOOK}, where the given proof is based on the
Cauchy integral formula.  The estimates of the form  $C'_f \le A C_f$ and $C_f\le B C'_f$,  where  $A,B>0$ are absolute constants, are also  known.

This classical result has counterparts for other mapping classes.     The theorem has a version for quasi-conformal mappings on uniform domains in $\mathbb{R}^n$ obtained by Astala and Gehring
\cite{ASTALA.MMJ}.   An analog for quasi-regular  mappings on Lip$_\alpha$-extension  domains is given  in the  Nolder paper  \cite{NOLDER.TAMS}.

Gehring and Martio \cite{GEHRING.CV} obtained a generalization of one part of the  Hardy and Littlewood theorem for (not necessarily analytic)      mappings on uniform domains, and    somewhat
later on Lip$_\alpha$-extension domains \cite{GEHRING.AASF}. In their work for (analytic) mappings on a  domain $D\subseteq\mathbb{C}$,          the  condition  \eqref{EQ.HL.1} is  replaced by
\begin{equation*}
|f'(z)| \le C_f d(z,\partial D)^{\alpha-1},\quad  z\in D,
\end{equation*}
where   $d(z,\partial D)$, is the distance  function, i.e.,
\begin{equation*}
d(z,B) = \inf_{x\in B} |x - z|,\quad z\in\mathbb{C},\quad B\subseteq\mathbb {C},  B\ne \emptyset.
\end{equation*}

Immediately after that,  Lappalainen in his Ph.D. thesis                 \cite{LAPPALAINEN.AASF}  gave a  generalization of the same  part of the Hardy and Littlewood theorem   for mappings on
Lip$_\varphi$-extension  domains,  where  $\varphi$  is a  majorant. A function  $\varphi\in C[0,\infty)$  is  called  the  majorant if it has the  following properties:

(1) $\varphi (0)  =  0$; $\varphi (t)>0$, $t\in (0,\infty)$;

(2) $\varphi$ is an increasing function on $(0,\infty)$;

(3) $\varphi'$ is a decreasing function on $(0,\infty)$.

\noindent These  properties imply   the following four (more important):

(a) $\varphi' (t)\le \frac {\varphi (t)}{t}$, $t\in (0,\infty)$;

(b) $\frac {\varphi (t)}{t}$ is decreasing;

(c) $\varphi (c t) \le c\varphi (t)$, $t\in (0,\infty)$,  $c> 1$;

(d) $\varphi(t+s) \le \varphi(t) + \varphi(s)$, $t,s\in (0,\infty)$.

A domain    $D$          is called a Lip$_\varphi$-extension domain \cite{LAPPALAINEN.AASF} if there exists a constant  $M$  such that for  every  $x, y\in D$ there  exists a rectifiable curve
$\gamma_{x,y}\subseteq  D$  with ends   at  $x$ and $y$ such that
\begin{equation}\label{EQ.LAPPALAINEN.COND}
\int _{\gamma_{x,y}} \frac{\varphi (d(z,\partial D))}{d(z,\partial D)} |dz|  \le M  \varphi ( |x-y|).
\end{equation}
For the standard majorant $\varphi(t)=\varphi_\alpha (t) = t^\alpha$, $\alpha\in (0,1]$,  one says     that $D$ is a Lip$_\alpha$-extension domain (the class of domains  introduced  previously
by Gehring  and Martio \cite{GEHRING.AASF}).

The variant of the Hardy and Littlewood theorem obtained by     Lappalainen \cite[Theorem 7.3]{LAPPALAINEN.AASF}      states that if for an analytic mapping on a Lip$_\varphi$-extension domain
$D\subseteq \mathbb{C}$ holds
\begin{equation}\label{EQ.LAPPALAINEN.1}
| f'(z) |\le  C_f\frac{\varphi (d (z,\partial D))}{d(z,\partial D)},\quad  z\in D,
\end{equation}
then $f$ satisfies
\begin{equation}\label{EQ.LAPPALAINEN.2}
|f(z) - f (w)| \le C'_f \varphi ( |z-w|),\quad z,w\in D,
\end{equation}
where  $ C_f$ and $ C'_f$  are constants.

The first condition displayed above is local,   and  the second one is of the global  nature.    In this  paper  we firstly examine, in the context of metric spaces, when does  a similar local
condition imply the corresponding global one. The condition \eqref{EQ.LAPPALAINEN.1} includes the distance function,  which we  shall replace  by a weight, that is a positive    and continuous
function, on a metric space. Then we consider the converse question, i.e., what can locally be said  about a mapping  which  satisfy  the condition \eqref{EQ.LAPPALAINEN.2} adapted  for metric
spaces. We  shall    generalize  separately the  two parts   of the classical theorem.   As  a consequence of our main results,  we are      able to derive the Hardy and Littlewood theorem for
analytic mappings  on the  unit ball   of a normed  space with values in an  another  normed  space.  

\section{Preliminaries}

This section contains preliminaries on mapping classes we shall encounter in the rest of this paper,                as well as  some facts concerning  the weighted distance on a metric   space.

As we have  already said,     an  everywhere positive and continuous function on  a  metric space  is called  the  weight  function.     Below  we introduce  the class of regularly oscillating
mappings between metric spaces with a weight  on the first one.

Let $(X,d_X)$ and $(Y, d_Y)$ be  metric  spaces.  For a  mapping $f:X\to  Y$  we define
\begin{equation*}
D^\ast f (x) = \limsup _{y\to x} \frac {d_Y ( f(x), f(y))}{d_X(x,y)},   \quad x\in X;
\end{equation*}
we  set $D^\ast  f (x) = 0$ for   an isolated $x\in X$.  The mapping  class $D^\ast(X,Y)$ contains $f:X\to Y$  such that   $D^\ast f (x)$ is finite for every  $x\in X$.

Assume moreover that $w$ be a weight on $X$. A mapping $f\in D^\ast(X,Y)$ which is bounded on every open ball $B_X(x,r)$, $x\in X$, $0<r<w(x)$,  is said to be regularly oscillating   if  there
exists   a  constant $K = K_f$ such that
\begin{equation}\label{EQ.RO}
D^\ast f  (x)  \le  \frac  {K} r {\sup_{y\in B_X(x,r)} d_Y ( f(x),  f(y) )},\quad x\in X, \quad 0<r<w(x).
\end{equation}
The corresponding    mapping class   is denoted by $\mathrm {RO}_w^K(X,Y)$.

In case of differentiable mapping between normed spaces  $D^\ast f(x)$ coincides with the operator norm of the differential $Df(x)$.   A mapping $f:D\to Y$,  where $D \subseteq X$ is a  domain
in a normed space $X$ (over $\mathbb{R}$ or $\mathbb{C}$), and $Y$  is another normed space, is  Fr\'{e}chet differentiable   at $x\in D$, if there exists a continuous linear  mapping $D f(x):
X\to Y$, called  Fr\'{e}chet  differential, such that in an open  ball  $B_X(x,r)=\{y\in X: \|x-y\|_X <r\}\subseteq D$ there holds
\begin{equation*}
f(y) -  f(x) = D f (x)  (y-x) +  \alpha (y-x),\quad y\in B_X(x,r),
\end{equation*}
where  $\alpha:B_X (0,r)\to  Y$ satisfies  $\lim _{h\to  0} \frac {\|\alpha (h)\|_Y}{\|h\|_X}  =  0$. We denote  by
\begin{equation*}
\|D f (x)\| = \sup _{\|\zeta\|_X = 1} \|D f(x)\zeta\|_Y
\end{equation*}
the  operator  norm of  $ D f(x): X\to Y$.  The  proof that $D^\ast f (x)$ is  equal   to the norm of Frech\'{e}t  differential $D f(x):X\to Y$  may be found in \cite{MARKOVIC.JGA}.  Therefore,
we have
\begin{equation*}
\| D   f (x)\|    =  \limsup_{y\to  x} \frac {\|f(x) - f(y)\|_Y}{\|x-y\|_X}.
\end{equation*}

If  $X$ and $Y$ are normed spaces over the field  $\mathbb {C}$, then a mapping $f:D\to Y$ is said  to be analytic on a domain $D\subseteq X$  if  it is  Fr\'{e}chet differentiable  everywhere
in   $D$.       The class  of regularly oscillating mapping on a domain in a normed with respect to the distance function of the domain contains   the class of bounded analytic mappings on $D$.
This may be easily  seen by using the  Schwarz lemma. For the sake of completeness we prove this in the lemma below.

\begin{lemma}
Every bounded analytic mapping $f:D\to Y$, where $D$ is domain a normed  space $X$,  and $Y$ is another normed space, belongs to the class  $\mathrm {RO}^1_{d}(D,Y)$, where $d$ is the distance
function on $D$, i.e.,  $d(z) = d(z,\partial D) = \inf_{w\in \partial D} \|z-w\|$,  $z\in D$.
\end{lemma}

\begin{proof}
Let $z\in D$.  For   $r<d(z, \partial D)$ the mapping  $g(\zeta ) = \frac 1M f(z + r\zeta)$,  where  $M = \sup_{w\in B(z, r)} \|f(w)\|_Y$, is analytic on $B_X(0,1)$ and bounded  by $1$. By the
Schwarz lemma there holds   $\|D g (0)\|\le 1$. Since  $D g(0) = \frac rM Df(z)$, we have
\begin{equation*}
\|Df (z) \| \le \frac {M}r = \frac 1 r \sup_{w\in B_X(z,r)} \|f(w)\|_Y.
\end{equation*}  If we replace above  $f$ with $f-C$, where  $C = f(z)$,  we obtain
\begin{equation*}
\|Df (z) \| \le    \frac 1 r \sup_{w\in B_X(z,r)} \|f(w) - f(z)\|_Y,\quad z\in D,\quad 0<r<d(z,\partial D),
\end{equation*}
which proves  that  $f$ belongs to  the class  $\mathrm{RO}^{1}_d (D,Y)$.
\end{proof}

\begin{remark}
Based on  \cite{PAVLOVIC.PA} and  \cite{PAVLOVIC.BOOK}  we give below some examples of classes of regularly oscillating functions on a domain in  $\mathbb{R}^n$  with respect to the   distance
function  of the domain  as  the weight  function.

The class $\mathrm{OC}^2(D)$, where  $D\subseteq \mathbb{R}^n$ is a domain,  contains  functions  $f\in C^2(D) $ which satisfy the inequality
\begin{equation*}
|\Delta f (z)|\le C_f \sup_{w\in B(z,r)}|\nabla  f(w)|,\quad z\in D,\quad 0<r<d(z,\partial D),
\end{equation*}
where $C_f$ is a constant. Every function in this class is regularly  oscillating  with respect to the distance function of the domain $D$. The  class $\mathrm {OC}^2 (D)$ contains many others
well known  classes  of  function, such as  harmonic, hyperbolically harmonic, and polyharmonic class of functions on the domain $D$.

On the other hand,         for a domain $D\subseteq \mathbb {R}^n$ one may relax the  definition of regularly oscillating  mappings in order to include more classes.   It may be asked that the
condition  \eqref{EQ.RO}  holds for a locally  Lipschitz mapping on $D$  and for almost every  $z\in D$. Then the class of regularly oscillating  mappings with respect to the distance function
on $D$  include, for example,    the  class of convex functions.
\end{remark}

We introduce now the mapping classes we need in this work.

For a non-negative function  $\varphi$ on $[0,\infty)$, which will be called the majorant in this context, we introduce the mapping class   $\Lambda_\varphi(X,Y)$  (a Lipschitz type class)  in
the following way: A mapping  $f:X\to Y$ belongs  to this class if  there  exists a constant  $C_f$ such that
\begin{equation*}
d_Y (f(x),f(y)) \le C_f \varphi (d_X (x,y)),\quad  x,y\in X.
\end{equation*}
The infimum of all $C_f$ is called the norm of $f$ in the class $\Lambda_\varphi (X,Y)$, and   denoted by $\|f\|_{\Lambda_\varphi (X,Y)}$. In the case $\varphi =  \varphi _ \alpha$, $\alpha\in
(0,1]$,   we write   $\Lambda ^\alpha (X,Y)$ instead of $\Lambda_{\varphi _\alpha}(X,Y)$.

For a weight $w$ on the metric space $(X,d_X)$ the class $\mathrm{B}_{w}(X,Y)$ (a Bloch-type class)  is consisted of mappings  $f\in D^\ast(X,Y)$  for which there exists a constant  $C_f$ such
that
\begin{equation*}
D ^\ast f(z) \le C_f w(z),\quad  z\in X.
\end{equation*}
The least  possible constant  $C_f$  is defined to be the  norm of  $f$ in this class, and denoted by  $\|f\|_{\mathrm {B}_{w} (X,Y)}$. 

Notice that  a mapping  $f:X\to  Y$ belongs to $\mathrm {B}_1 (X,Y)$  if and  only if  $D^\ast f$ is bounded on $X$.

Finally,      for the sake of completeness we include here some facts concerning the weighted distance on a metric space. We refer to \cite{BJORNS.BOOK} for the detailed exposition;   see also
\cite{MARKOVIC.JGA} for  a  short one.

A continuous mapping  $\gamma:[a,b]\to X$, where $[a,b]\subseteq \mathbb{R}$ is a segment, is called the curve in a metric space $(X,d_X)$ (we  identify the curve $\gamma$       with its image
$\gamma([a,b])\subseteq X$).  We say that the curve $\gamma$ connects $\gamma (a)$ and $\gamma (b)$, or that  $\gamma(a)$ and $\gamma(b)$ are endpoints of  the curve $\gamma$. The    length of
the   curve  $\gamma$ is
\begin{equation*}
\ell  (\gamma)   = \sup   \sum_{i=1}^n d_X  (\gamma(t_{i-1}),\gamma (t_i)),
\end{equation*}
where the supremum is taken  over all partition $a = t_0 < t_1 < \dots < t_n = b$, $n\in\mathbb{N}$, of $[a,b]$. A  curve  is rectifiable  if its  length  is finite.  The family of rectifiable
curves with endpoints at $x$ and $y$ is denoted by $\Gamma ({x,y})$.          A metric space  $X$ is rectifiable connected if  $\Gamma (x,y) $ is not empty for every $x,y\in X$. Obviously, for
$\gamma\in\Gamma _{x,y}$   we  have $d_X (x,y)\le\ell (\gamma)$.  If $\gamma:[a,b]\to X$ is a rectifiable curve in a metric space $X$,    and    $[c,d] \subseteq [a,b]$, then the (rectifiable)
curve  $\gamma:[c,d]\to X$  is denoted  by  $\gamma_{[c,d]}$. Clearly, we have $\ell (\gamma) = \ell  (\gamma_{[a,c]}) + \ell(\gamma_{[c,b]})$, $c\in [a,b]$.

The  integral of a continuous function    $f$ on $X$  over a  rectifiable  curve $\gamma\subseteq X$  is denoted by  $\int_\gamma  f$.  It is the  number which satisfies:             For every
$\varepsilon>0$ there  exists $\delta>0$ such that
\begin{equation*}
\left|\int_\gamma  f   -  \sum _{i=1}^n f (\gamma (s_i)) \ell(\gamma|_{[t_{i-1},t_i]} ) \right|<\varepsilon
\end{equation*}
for every partition of $t_0,t_1,\dots,t_n$ of $[a,b]$ with $\max_{1\le i\le n}|t_i - t_{i-1}|<\delta$, and every sequence $s_1,s_2,\dots,s_n$  such that $t_{i-1}\le s_i\le t_i$, $1\le i \le n$.

Let $w$ be a weight on a metric space $X$.  The weighted  distance on  the  (rectifiable connected)  metric  space $X$ is introduced  by
\begin{equation*}
d_w   (x, y)  =  \inf_{\gamma \in \Gamma_{x,y}} \int_\gamma w, \quad x,y\in X.
\end{equation*}
In particulary  for $w\equiv 1$ we have $\int_\gamma 1 = \ell(\gamma)$ and  $d_1 (x,y) =  \inf_{\gamma \in \Gamma_{x,y}}  \ell(\gamma)$, $x,y\in X$. The distance $d_{1}$ is in some cases equal
to  $d_X$. For example, this occurs if  $X$ is a convex domain in  a normed space. However, in general, we have  $d_X \le d_1$.    In case of domains in $\mathbb{R}^n$,  $d_1$ is known  as the
inner distance.        In case when the weight function is  reciprocal  to the  distance function of a domain,  for the weighted  distance we have the known   quasi-hyperbolic distance  on the
domain   introduced  by Gehring and Palka.

\begin{remark}
The topology  of $(X, d_X)$ may be the  same as the topology of the metric space  $(X, d_w)$.                             Let   $(X, d_X)$   be a metric space with the following local property
\begin{equation}\label{EQ.COND.LOCAL}
(\exists M>0)(\forall x\in X) (  \exists r>0) ( \forall  y\in B_X(x,r) ) ( \exists \gamma \in  \Gamma_{x,y}) (  \ell (\gamma_{x,y} ) \le M d(x,y)).
\end{equation}
For a    weight $w$  on $X$ we  shall   prove
\begin{equation}\label{EQ.LIMINFSUP.W}
 w (x)\le \liminf _{y\to  x} \frac { d_w(x,y)} {d_X(x,y)} \le \limsup  _{y\to  x} \frac { d_w(x,y)} {d_X(x,y)} \le  M  w (x), \quad x\in X,
\end{equation}
which is  sufficient  for the  equality of the two  topologies.

Indeed,  since   $w$  is continuous, for   $\varepsilon> 0$  there exists an open ball $B_X(x,r)$ such that
\begin{equation*}
0<w(x)-\varepsilon<w (y)<w(x)+\varepsilon,\quad y\in B_X(x,r).
\end{equation*}
We may replace $r$ with a smaller number such that \eqref{EQ.COND.LOCAL} is satisfied.  Now,  we have
\begin{equation*}
d_w (x,y)  \le \int_{\gamma _{xy}}  \omega  \le   (w (x) +\varepsilon )\ell (\gamma_{xy}) \le     (w (x) +\varepsilon ) M d_X(x,y).
\end{equation*}
On the other hand, if $\gamma$ is among curves that connect   $x$  and $y$, then
\begin{equation*}
d_{w} (x,y) = \inf_\gamma \int_\gamma w\ge (w (x) - \varepsilon) d_X(x,y).
\end{equation*}
Therefore,
\begin{equation*}
w(x) - \varepsilon  \le \frac{d_w (x,y)}{d_X (x,y)} \le M(w(x) + \varepsilon),\quad    y\in B_X(x,r),
\end{equation*}
which implies
\begin{equation}\label{EQ.LIM.W}
w(x) - \varepsilon \le \liminf _{y\to  x} \frac { d_w(x,y)} {d_X(x,y)} \le \limsup  _{y\to  x} \frac { d_w(x,y)} {d_X(x,y)}\le M ( w (x)  + \varepsilon ), \quad x\in X.
\end{equation}
Letting  $\varepsilon\to 0$, we obtain \eqref{EQ.LIMINFSUP.W}.

The property \eqref{EQ.LIMINFSUP.W} occurs on domains in  normed spaces, since  one may chose  $\gamma_{x,y} (t)=(1-t) x + ty$, $t\in [0,1]$. It is easy to show that   $\gamma_{x,y}$      is a
rectifiable curve, and   $\ell (\gamma _{x,y}) = \|x-y\|$.

For similar results on the coincidence   of topologies we  refer to  \cite{ZHU.JLMS}.
\end{remark}

\section{Main results}

This section contains our main results.       At the end of the paper, based on our approach, we prove a version of the Hardy and Littlewood theorem for analytic mappings between normed spaces.

The following theorem generalize the one part of the classical Hardy and Littlewood theorem. It contains an integral condition which connects the weight  on  a  metric space  with the ordinary
metric through the majorant. This condition is a  generalization of \eqref{EQ.LAPPALAINEN.COND} which concerns  domains in $\mathbb{R}^n$,  and it simply means that the  weighted  distance may
be controlled by the ordinary one.

\begin{theorem}\label{TH.1}
Let $(X,d_X)$ and $(Y,d_Y)$  be  metric spaces,  $w$   a weight  on  $X$, and  $\varphi$ a  non-negative function  on $[0,\infty)$.                               Assume the following condition:
\begin{equation}\label{EQ.COND.LAPP}
(\exists M>0)(\forall x,y\in X )( \exists \gamma \in \Gamma(x,y)) \int _\gamma w \le M  \varphi  (d_X (x,y)).
\end{equation}
Then we have
\begin{equation*}
\|f\| _{ \Lambda_\varphi  (X,Y)}\le   M\|f\| _{ \mathrm {B}_w (X,Y)}, \quad f\in \mathrm {B}_w (X,Y).
\end{equation*}
In particular,   if $\varphi$ is continuous in $0$ and $\varphi(0)=0$, then every  $ f\in \mathrm {B}_w (X,Y)$ has a continuous extension on the  minimal completion of the metric space $X$, if
$Y$ is a complete metric space.
\end{theorem}

In the proof we use the following lemma  which gives the estimate of the length of the image of a rectifiable curve. It  may  be found  with a  proof in \cite{MARKOVIC.JGA}.

\begin{lemma}[see \cite{MARKOVIC.JGA}]\label{LE.IMAGE}
Let $X$ and $Y$  be metric spaces,  $\gamma\subseteq X$  a rectifiable curve, and $f \in \mathrm{B}_1 (X, Y)$. Then the curve  $f \circ \gamma \subseteq Y$  is also  rectifiable,  and  we have
\begin{equation*}
\ell (f \circ \gamma ) \le \|f\|_{ \mathrm {B}_1 (X,Y)} \ell(\gamma).
\end{equation*}
\end{lemma}

\begin{proof}[Proof of Theorem \ref{TH.1}]
Let $f\in  \mathrm {B}_w (X,Y)$ satisfies
\begin{equation}\label{EQ.DSTAR}
D ^\ast  f(z) \le C_f w(z), \quad z\in X,
\end{equation}
where $C_f $ is a constant. We shall show  that   $f$ belongs to the class   $\Lambda_\varphi (X,Y)$,  and
\begin{equation*}
d_Y(f(x), f(y))\le  {C}' _f \varphi (d_X(x,y)),\quad x, y\in X,
\end{equation*}
where  $ {C}'_f = M C_f $. This clearly    implies the  statement in our theorem.

Since of \eqref{EQ.DSTAR}  the quantity
\begin{equation*}
B_f  = \|f\|_{\mathrm {B}_w(X,Y)} = \sup_{z\in X} {w(z)} ^{-1}  D^\ast  f(z)
\end{equation*}
is finite.      Note that $D^\ast f$ is bounded on every compact subset of $X$, which also follows from \eqref{EQ.DSTAR} (since $w$ is continuous). Therefore, $f\in \mathrm{B}_1(K,Y)$,   where
$K\subseteq X$ is a  compact set.

Let $\gamma:[a,b]\to X$ be a rectifiable curve such that $\gamma(a) = x$ and $\gamma(b) = y $ (it follows from \eqref{EQ.COND.LAPP} that the set $\Gamma (x,y)$ is not empty). The  image  curve
$f \circ \gamma$ connects  $f(x)$  with  $f(y)$, and   by Lemma \ref{LE.IMAGE} this curve is  also  a rectifiable one.       Therefore,  $\ell (f \circ \gamma) =\int _ {f \circ \gamma} {1}$ is
finite (which means that the set  $\Gamma (f(x),f(y))$ is also non-empty). We now  proceed to estimate the  last  term.

Our first aim is to prove that the mapping    $f$   is Lipschitz with respect to the weighted distance on the first metric space,     and the inner distance on the second metric space,    i.e.,
\begin{equation}\label{EQ.LIP.B}
d_1 (f(x),f(y)) \le B_f  {d_w( x,y)},\quad  x, y\in X.
\end{equation}

Let $\varepsilon> 0$ be arbitrary.         Since $w(z)$  is continuous,              and since a  curve is a compact set, we may found a partition $a= t_0<t_1<\dots < t_n = b$ of $[a,b]$  with
$\max_{1\le i\le n} |t_{i-1} -t_i| <\varepsilon$,  such that
\begin{equation*}
|w(\gamma (t)) -  w(\gamma (t_{i-1} ) )|<B^{-1}_f\varepsilon , \quad t \in [t_{i-1},t_{i}],\quad 1\le  i\le  n.
\end{equation*}
Then we have
\begin{equation*}
w(\gamma (t)) < w(\gamma (t_{i-1})) +  B^{-1}_f\varepsilon,
\end{equation*}
so it follows
\begin{equation*}
D^\ast  f (\gamma (t))\le  B _f  w (\gamma (t)) < B_f  w (\gamma (t_{i-1}))+\varepsilon.
\end{equation*}
By Lemma \ref{LE.IMAGE} we have
\begin{equation*}
\ell(f\circ \gamma|_{[t_{i-1},t_i]})\le (B_f w(\gamma (t_{i-1}))+\varepsilon)  \ell(\gamma|_{[t_{i-1},t_i]}),
\end{equation*}
therefore  for the corresponding   integral sum  of   $\int_{f\circ \gamma } 1$  we derive the  following estimate
\begin{equation*}\begin{split}
\sum _ {i=1}^n  1\cdot  \ell(f\circ \gamma|_{[t_{i-1},t_i]})
&\le\sum_{i=1}^n (B_f w(\gamma (t_{i-1}))+\varepsilon)  \ell(\gamma|_{[t_{i-1},t_i]})
\\& \le \sum _ {i=1}^n   B_f w(\gamma (t_{i-1})) \ell(\gamma|_{[t_{i-1},t_i]})
+ \varepsilon \sum _ {i=1}^n   \ell (\gamma|_{[t_{i-1},t_i]})
\\& \le \sum _ {i=1}^n   B_f w(\gamma( t_{i-1} )) \ell(\gamma|_{[t_{i-1},t_i]})
+ \varepsilon   \sum _ {i=1}^n    \ell(\gamma|_{[t_{i-1},t_i]})
\\& \le B_f  \sum _ {i=1}^n   {w(\gamma( t_{i-1} ))} \ell(\gamma|_{[t_{i-1},t_i]})  + \varepsilon  \ell(\gamma).
\end{split}\end{equation*}
Letting  $\varepsilon\to 0$  above, we obtain
\begin{equation*}
\int_{f\circ \gamma} 1 \le B_f   \int_\gamma w\Longrightarrow \inf_{\gamma\in \Gamma(x, y)} \int_{f\circ \gamma} 1 \le  B_f   d_ w (x,y),
\end{equation*}
so it follows
\begin{equation*}\begin{split}
d_{1} (f(x),f(y))& = \inf _{\delta\in \Gamma(f(x),f(y))} \int _ {\delta} { 1 }
\\& \le\inf_{\gamma\in \Gamma (x, y)}  \int_{f\circ \gamma}1\\&\le B_f  d_w (x, y),
\end{split}\end{equation*}
which proves  \eqref{EQ.LIP.B}.

We shall now  prove that $f$ belongs to the class  $\Lambda_\varphi (X,Y)$.        The  condition  \eqref{EQ.COND.LAPP} implies   $d_{w} \le M (  \varphi \circ d_X)$. Since $d_Y \le d_1$   and
$B_f\le C_f$, for every   $x,y\in X$ we have
\begin{equation*}\begin{split}
d_Y (f(x), f(y)) & \le d_1(f(x),f(y)) \\&\le  C_f d_{w} (x,y) \\& \le  M C_f   \varphi ( d_X (x,y) ),
\end{split}\end{equation*}
which is the inequality we aimed to prove.
\end{proof}

The next two theorems are converse of the preceding one but for regularly oscillating mappings with respect to the weight on the first metric space and some  conditions on the majorant. In the
first one theorem we have very strong assumptions on the majorant  in order that we  obtain  the control of  the B-norm.    In the  second one we  relax    the conditions,  but we have certain
assumption  on  the  first  metric space along  with the  weight.  However, in  the last case we are not able to obtain   any control of the  norm of a mapping in $\mathrm {B}$-space  via  the
norm of  the same  mapping   in $\Lambda$-space.      In the proof of the last mentioned result we use an idea of Stegbuchner  \cite{STEGBUCHNER.AASF}.

\begin{theorem}\label{TH.2.1}
Let  $(X,d_X)$  and  $(Y,d_Y)$ be  metric spaces, and let  a function $\varphi \in C^1 (0,\infty) $ be nonnegative on $[0,\infty)$ with the following property:  There exists  a constant  $A>0$
such that
\begin{equation}\label{EQ.COND.A}
\frac  {\varphi (t)} {t} <A \varphi'(t), \quad t\in (0,\infty).
\end{equation}
Then      we have
\begin{equation*}
\|f\| _{\mathrm {B} _{\varphi'\circ w} (X,Y)}\le A K \|f\|_{\Lambda_\varphi(X,Y)}, \quad f\in \Lambda_\varphi(X,Y)\cap \mathrm {RO}^K_w(X,Y)
\end{equation*}
for every $K>0$ and   every weight  $w$ on the metric space $X$.
\end{theorem}

\begin{proof}
Assume   that   $f\in \mathrm {RO}^K_w(X,Y)$ belongs   to  $\Lambda_\varphi  (X,Y)$, and let us  denote $ C_f=\|f\|_{\Lambda_\varphi  (X,Y)}$.  Then  we have
\begin{equation*}
d_Y (f(x),f(y)) \le C_f \varphi (d_X (x,y)),\quad  x,y\in X.
\end{equation*}
We shall show that    $f\in B _ {\varphi'\circ w} (X,Y)$ and
\begin{equation*}
D^\ast f (z) \le   AKC_f\varphi ' (w(z)),\quad  z\in X,
\end{equation*}
from which follows  the conclusion of  our theorem.

Denote  $M = AKC_f$,  and    assume  contrary, i.e., that there exists  $z\in X$ such that
\begin{equation}\label{EQ.M.CONTRA}
D^\ast f (z) >  M   \varphi ' (w(z)).
\end{equation}
Since   $f\in\mathrm { RO}^K_w(X,Y)$,   we have
\begin{equation*}
D^\ast f (z) \le \frac{K}r \sup_{x\in B(z,r)} d_Y ( f(z), f(x)), \quad  0< r < w(z),
\end{equation*}
which implies
\begin{equation*}
\sup_{x\in B(z,r)} d_Y ( f(z), f(x))\ge \frac r{K} D^\ast f (z)  >  \frac {rM}{K}   \varphi ' (w(z)).
\end{equation*}
On the other hand,  since $f\in \Lambda _\varphi (X,Y)$ and    since  $\varphi$   is increasing, we   obtain
\begin{equation*}\begin{split}
\sup_{x\in B(z,r)} d_Y ( f(z), f(x))\le C \sup_{x\in B(z,r)}  \varphi (d_X (z,x))  \le C \varphi (r), \quad  0< r< w(z).
\end{split}\end{equation*}
Therefore,   for every $r \in (0,w(z))$  the following  inequality holds
\begin{equation*}
\frac {rM}{K}   \varphi ' (w(z))<  \sup_{x\in B(z,r)} d_Y ( f(z), f(x))\le  C \varphi (r).
\end{equation*}
It follows that we must have
\begin{equation*}
A \varphi ' (w(z))<  \frac {\varphi (r)}r, \quad r \in (0,w(z)).
\end{equation*}
However, according to our assumption on $\varphi$, this is not possible. Indeed, by   the condition  \eqref{EQ.COND.A} we have
\begin{equation*}
 \frac {\varphi (w(z))}{w(z)} < A  \varphi ' (w(z)).
\end{equation*}
Now, since  of continuity of $\varphi'$,   we  obtain  that if $r<w(z)$ is sufficiently close to $w(z)$  the following   inequality holds
\begin{equation*}
\frac {\varphi (r)}{r} < A  \varphi ' (w(z)).
\end{equation*}
Since we reached  a contradiction, we conclude that \eqref{EQ.M.CONTRA} cannot  hold for any $z\in X$. This proves the inequality in our theorem.
\end{proof}

\begin{corollary}\label{CORO.HL.ALPHA}
Let $(X,d_X) $ and $(Y,d_Y)$ be metric spaces,  $\alpha \in (0,1)$, and  $w$  a  weight  on $X$. Assume  that
\begin{equation}\label{EQ.COND.ALPHA}
(\exists M>0)(\forall x,y\in X)( \exists \gamma \in \Gamma (x,y) ) \int _{\gamma}  w^{\alpha-1} \le M d_X (x,y) ^\alpha.
\end{equation}
Then we have
\begin{equation*}
\mathrm{RO}_w^K(X,Y) \cap  \Lambda^\alpha  (X,Y )  =  \mathrm{RO}_w^K(X,Y) \cap   \mathrm{ B}_{  w ^{\alpha-1}}(X,Y),
\end{equation*}
Moreover,   a   regularly oscillating  mapping  $f\in \mathrm {RO}_w^K(X,Y)$ satisfies
\begin{equation*}
D ^\ast f(z) \le C_f w(z)^{\alpha-1}, \quad z\in X,
\end{equation*}
if and only if it satisfies
\begin{equation*}
d_Y(f(x), f(y))\le  C'_f  d_X(x,y)^\alpha,\quad x, y\in X,
\end{equation*}
where $C_f$ and $C'_f$ are constants.                         If  the first inequality holds, then we can take $C'_f\le M C_f$. If the second inequality holds,    then the first one holds with
$C_f \le  \frac {2K}{\alpha}  C'_f$.
\end{corollary}

\begin{proof}
We have  just  to  apply Theorem \ref{TH.1} and Theorem \ref{TH.2.1}  for the majorant  $\varphi _{\alpha}$ and the constant  $A =\frac 2\alpha $.
\end{proof}

As a consequence of the above result we have the following statement for analytic mappings between normed spaces.  We note that the classical   theorem of Hardy and Littlewood remains with the
same constants estimate,  if we replace analytic mapping on the unit disk with analytic mapping on the unit  ball  in a  normed space  with values in another normed  space.

\begin{corollary}
Let $f$ be an  analytic mapping  on $B_X (0,1)$,  which is the unit ball in a normed  space $X$, with the  range  in  another normed space $Y$. Then we have
\begin{equation*}
\|Df (z)\| \le C_f (1-\|z\|_X)^{\alpha-1}, \quad z\in B_X(0,1),
\end{equation*}
if and only if $f$ satisfies the  H\"{o}lder condition
\begin{equation*}
\|f(z) - f(w)\|_Y\le C'_f\|z-w\|_X^{\alpha},\quad  z,w\in B_X (0,1),
\end{equation*}
where $C_f$ and $C'_f$ are constants.  Moreover, there holds the estimates  $C'_f\le\frac 4\alpha C_f$  and   $C_f\le \frac 2\alpha C'_f$.
\end{corollary}

\begin{proof}
That the  unit ball  $B_X(0,1)$ in the Hilbert space  $X$  satisfies the condition of the preceding theorem with the constant    $M = \frac 4\alpha$,   may be proved  as in the work of Gehring
and Martio \cite{GEHRING.AASF} at  the  page 205. For $x,y\in B_X(0,1)$,         as in the cited work,  one finds a curve $\gamma\in  \Gamma(x,y)$  such that
\begin{equation*}
\int_{\gamma} d^{\alpha-1}\le \frac{\pi}{\alpha 2^\alpha}\|x-y\|^{\alpha}.
\end{equation*}
where $d$ is the distance function (from the boundary of $B_X(0,1)$). The constant which appears above  is   $\le \frac {4}\alpha$, so it follows  the first inequality between  constants.  For
the reverse estimate we only have to note that  $f$  belongs  to the class $\mathrm {RO}^1_d(B_X(0,1), Y)$  and apply the preceding corollary.
\end{proof}

\begin{theorem}\label{TH.2.2}
Let     $(X,d_X)$       be a    $\sigma$-compact metric space,    $w$  a weight on $X$ such that for  every $\varepsilon>0$   there exists a compact set                $K = K_\varepsilon$ with
\begin{equation*}
w(z) < \varepsilon ,\quad  z\in X\backslash K,
\end{equation*}
and  $\varphi \in C[0,\infty)$   continuously differentiable and concave on   $(0,\infty)$,     with  $\varphi(0)=0$ and  $\varphi'(t)>0$, $t\in(0,\infty)$.  If a regularly oscillating mapping
$f:X\to Y$, where $(Y,d_Y)$ is another metric space, belongs to  $\Lambda_\varphi (X,Y)$, then there exists  a constant  $C$  such that
\begin{equation*}
D^\ast f (z) \le C \varphi ' (w(z)),\quad  z\in X,
\end{equation*}
i.e., $f\in B_{\varphi ' \circ w}$.
\end{theorem}

\begin{proof}
Let us denote
\begin{equation*}
Q(z)  =  \frac{D^\ast f(z)}{\varphi'(w(z))},\quad z\in X.
\end{equation*}
We shall  firstly   show  that  $Q$ is bounded on every  compact set  $K_1 \subseteq X$.

Assume  contrary,  i.e.,   that  a compact  set  $K_1$  contains a sequence $(z_n)_{n\in\mathbb{N}}$  such that  $(Q(z_n))_{n\in\mathbb {N}}$ diverge   to infinity.  We can select a convergent
subsequence of  $(z_n)_{n\in\mathbb{N}}$. Denote it in the same way, and let it  converge to $z\in K_1$.  Let $r = \frac { w (z)}2$. If  an integer $n$ is enough big, we  have  $z_n\in B(z,r)$.
Since  $\lim_{n\to \infty} w(z_{n}) =  w(z)$, we have
\begin{equation*}
\lim_{n\to \infty}\varphi'(w(z_{n})) =  \varphi'( w(z))> \frac { \varphi'( w(z))}2.
\end{equation*}
Therefore, if $n$ is enough big,  there holds   $\varphi'(w(z_{n})) > \frac { \varphi'( w(z))}2$.  Let us denote
\begin{equation*}
r_n = \min\left\{\frac { w(z_n) }2,\frac r2\right\} = \min\left\{\frac { w(z_n) }2, \frac { w(z) }4\right\} \ge \frac  m4,
\end{equation*}
where $m= \min_{\zeta \in K} w(\zeta)$.   Since $f\in \mathrm {RO}^K_w(X,Y)$, we have
\begin{equation*}\begin{split}
Q(z_n) & = \frac{D^\ast f(z_{n})}{\varphi'(w(z_{n}))} \le \frac{D^\ast f(z_{n})}{\frac {\varphi'(w(z))}2}
\\&\le \frac {2}{{\varphi'(w(z))}} \frac K {r_n}\sup_{x\in B(z_{n},r_n)} d_Y (f(x),  f(z_{n}))
\\& \le \frac {2K}{\varphi'(w(z))} \frac1{r_n} \sup_{x\in B(z_{n},r_n)} (d_Y (f(x), f(z)) + d_Y(f(z),  f(z_{n})))
\\&\le \frac {2K}{\varphi'(w(z))} \frac 1{\frac m4}2\sup _{y\in {B}(z,r)} d_Y (f(y), f(z)),
\end{split}\end{equation*}
which is  a  contradiction, since  the  sequence  $(Q(z_n))_{n\in\mathbb{N}}$  is unbounded.

If $X$ is a compact metric space,    we may finish the proof at this moment. In the sequel we suppose that this is not the case,            and  let $X = \bigcup_{n\in \mathbb{N}} K_n$,  where
$(K_n)_{n\in \mathbb{N}}$   is a strictly   increasing sequence of  compact sets  in  the sense that         $K_1\subset K_2\subset \cdots \subset K_n\subset K_{n+1}\subset \cdots$

Assume that there exists a sequence              $(z_n)_{n\in \mathbb{N}}$ in $X$ such that  $\lim _{n\to \infty} Q (z_n) = \infty$,          and for  the sake  of  simplicity,  let  us denote
\begin{equation*}
Q_n = Q(z_n)  = \frac{D^\ast f(z_n)}{\varphi'(w(z_n))},\quad n\in \mathbb{N}.
\end{equation*}
Since  $Q (z)$  is bounded on every  compact subset  of  $X$,   the  set  $K_m$, $m\in \mathbb{N}$, may contain only a finite members of the sequence $(z_n)_{n\in\mathbb{N}}$, so  according to
our  assumption  on the weight   $w$,  we must  have  $\lim_{n\to\infty } w(z_n) = 0$.

We will construct inductively a subsequence $(z_{n_k})_{k\in\mathbb{N}}$ of $(z_n)_{n\in\mathbb{N}}$: Let $n_1$ be such that $Q_{n_1}\ge 4$.             Assume that we have found $n_k$. Denote
$\delta_k = \frac { w (z_{n_k})}2$.  Let $\delta_{k+1}^\ast\in (0,\delta_k)$ be  the unique  number  such that
\begin{equation*}
\varphi (\delta _{k+1} ^\ast) =  \frac {\varphi (\delta _k)}2
\end{equation*}
(which exists,   since  $\varphi$ is continuous and  increasing on $[0,\infty)$).            Now,  let $n_{k+1}> n_k$ be such that $0< w (z_{n_{k+1}}) < \delta _{k+1}$  (since the     sequence
$(w(z_n))_{n\in \mathbb{N}}$ converge to $0$, it is possible to  find a such member), and  $Q_ {n_{k+1}} \ge 2^{k+2}$     (recall that $(Q_n)_{n\in\mathbb{N}}$ diverge to $\infty$).

For the sake of simplicity  let the subsequence $(z_{n_k})_{k\in\mathbb{N}}$ be  denote by $(z_k )_{k\in \mathbb{N}}$. By the construction of the sequence $(z_k)_{k\in \mathbb{N}}$     we have
\begin{equation*}
D^\ast f (z_{k}) \ge 2^{k+1} \varphi' (w (z_{k})),\quad \varphi (\delta_{k+1})\le \frac{ \varphi (\delta_k)}2, \quad k\in \mathbb{N}.
\end{equation*}

Since  $f$ belongs  to the class $\mathrm{RO}_w^K(X,Y)$, and since $\delta_k< w(z_k)$, we have
\begin{equation*}
D^\ast f (z_{k}) \le \frac  K {\delta _k} \sup_{x\in B(z_k,\delta_k)} d_Y (f(x),  f(z_{k})).
\end{equation*}
Since    $\varphi$  is concave on $(0,\infty)$, we have
\begin{equation*}\label{EQ.CONCAVE}
\varphi'(t)\ge \frac{\varphi(s) - \varphi(t)}{s-t},\quad s>t>0.
\end{equation*}
Therefore,
\begin{equation*}\begin{split}
\sup_{x\in B(z_k,\delta_k)}d_Y (f(x), f(z_{k}))& \ge\frac{\delta_k}{K} D^\ast f (z_{k}) \\& \ge\frac{\delta_k}{K} 2^{k+1} \varphi' (\delta_k) 
\\& \ge \frac{\delta_k}{K}  2^{k+1}  \frac{\varphi(\delta_k) - \varphi (\delta_{k+1})}{\delta_k - \delta_{k+1}}
\\&\ge \frac { 2^{k+1}} K \frac{\delta_k}{\delta_k - \delta_{k+1}}  (\varphi(\delta_k) - \varphi (\delta_{k+1})) 
\\&\ge \frac {2^{k+1}} K   (\varphi(\delta_k) - \varphi (\delta_{k+1})) 
\\& = \frac {2^{k+1}} K (\varphi(\delta_k) - \frac 12 \varphi (\delta_{k}))\\& =  \frac  {2^{k}}K \varphi (\delta_k).
\end{split}\end{equation*}
However, if  $k$ is  enough large  integer, the just  obtained inequality is opposite to the following one.

Since  $f\in \Lambda_\varphi (X,Y)$ (recall that $\varphi$ is an  increasing function) we have
\begin{equation*}\begin{split}
\sup_{x\in B(z_k,\delta_k)}d_Y (f(x),  f(z_{k})) &  \le \|f\|_{\Lambda_\varphi (X,Y)} \sup_{x\in B(z_k,\delta_k)} \varphi (d_X (x,z_k))
 \\&  \le   \|f\|_{\Lambda_\varphi (X,Y)}  \varphi (\delta_k).
\end{split}\end{equation*}

We have just reached a contradiction. Therefore,   the sequence         $(z_n)_{n\in \mathbb{N}}$        for which            $(Q(z_n))_{n\in\mathbb{N}}$              diverge does   not  exist.
\end{proof}

\subsection*{Acknowledgements}
The author             is        thankful   to  the referee of this paper for useful advices,    for carefully reading,      and finding  several errors in the first version of the  manuscript.

\end{document}